\documentclass[11pt,a4paper]{article}
\usepackage{amsmath,amsfonts,amssymb,amsthm}
\usepackage{bbm,mathrsfs,bm,mathtools}

\title{Another look at threshold phenomena\\ for random cones}

\author{Daniel Hug and Rolf Schneider}

\date{}
\newcommand{\R}{{\mathbb R}}
\newcommand{\bP}{{\mathbb P}}

\newcommand{\N}{{\mathbb N}}

\newcommand{\bE}{{\mathbb E}\,}

  \renewcommand{\exp}{{\rm exp}\,}

  \newcommand{\D}{{\rm d}}

\newtheorem{theorem}{Theorem}
\newtheorem{proposition}{Proposition}

\begin{document}
\maketitle

\begin{abstract}
In stochastic geometry there are several instances of threshold phenomena in high dimensions: the behavior of a limit of some expectation changes abruptly when some parameter passes through a critical value. This note continues the investigation of the expected face numbers of polyhedral random cones, when the dimension of the ambient space increases to infinity. In the focus are the critical values of the observed threshold phenomena, as well as threshold phenomena for differences instead of quotients. \\[1mm]
\end{abstract}

\section{Introduction}\label{sec1}

In stochastic geometry, it is known that several expectations of geometric functionals exhibit a threshold phenomenon when the dimension tends to infinity. This note collects some observations connected to threshold phenomena for random convex cones.

Let $\phi$ be a probability measure on $\R^d$ which is even and assigns measure zero to each hyperplane through the origin $o$. Let $X_1,\dots,X_N$ with $N\ge d+1$ be stochastically independent random vectors with distribution $\phi$.
The Donoho--Tanner random cone is defined by $D_N:= {\rm pos}\{X_1,\dots,X_N\}$, where pos denotes the positive hull.
For $k\in\{1,\dots,d-1\}$ and any polyhedral convex cone $C\subset\R^d$ we denote by $f_k(C)$ the number of $k$-dimensional faces of $C$. Thus, $f_k(D_N)\le\binom{N}{k}$, and the equality holds if and only if $D_N$ is $k$-neighborly, that is, the positive hull of any $k$ of its generating vectors $X_1,\dots,X_N$ is a $k$-face of $D_N$. Now we assume that the dimension $d$ increases to infinity, while $N=N(d)$ and $k=k(d)$ grow with the dimension. (In the following, the dependence on $d$ will not always be shown by the notation.) We ask for the asymptotic behavior of the quotient $\bE f_k(D_N)/\binom{N}{k}$, where $\bE$ denotes the expectation.

Let $0<\delta<1$ and $0\le\rho<1$ be given. Let $k<d<N$ be integers (depending on $d$) such that
\begin{equation}\label{1A.1}
\frac{d}{N} \to\delta, \qquad \frac{k}{d}\to \rho \qquad \mbox{as }d\to\infty.\end{equation}
Define $\rho_W(\delta):= \max\{0,2-\delta^{-1}\}$ for $0<\delta<1$. Then
\begin{equation}\label{1A.2}
\lim_{d\to\infty} \frac{\bE f_k(D_N)}{\binom{N}{k}} =\left\{ \begin{array}{ll} 1 & \mbox{if }\rho<\rho_W(\delta),\\ 0 & \mbox{if }\rho >\rho_W(\delta).\end{array}\right.\end{equation}
This was proved by Donoho and Tanner \cite{DT10}. (In their terminology, $D_N={\sf A}\R^N_+$, where ${\sf A}$ is the random matrix with columns $X_1,\dots,X_N$ and $\R^N_+$ is the nonnegative orthant of $\R^N$.)

A similar result holds for the Cover--Efron cones, which are random cones with a different distribution, introduced earlier (see Cover and Efron \cite{CE67}). With random vectors $X_1,\dots,X_N$ as above, let $C_N$ be the random cone defined as ${\rm pos}\{X_1,\dots,X_N\}$, under the condition that this positive hull is different from $\R^d$. The following  was proved by Hug and Schneider \cite{HS20}. If (\ref{1A.1}) holds, then
\begin{equation}\label{1A.3}
\lim_{d\to\infty} \frac{\bE f_k(C_N)}{\binom{N}{k}} =\left\{ \begin{array}{ll} 1 & \mbox{if }\rho<\rho_W(\delta),\\ 0 & \mbox{if }\rho >\rho_W(\delta).\end{array}\right.
\end{equation}
The similarity to (\ref{1A.2}) is partially plausible, due to the fact that the probability of the event $\{{\rm pos}\{X_1,\dots,X_N\}\not=\R^d\}$ becomes exponentially small as $d\to \infty$, but only partially (the proof is more complicated).

If the number $k$ remains fixed, then we no longer have such a sharp threshold as above. The following was also proved in \cite{HS20}. If $\frac{d}{N} \to\delta \quad\mbox{as }d\to\infty$ with a number $\delta\in[0,1]$ and if $k\in\N$ is fixed, then
\begin{equation}\label{1A.4}
\lim_{d\to\infty} \frac{\bE f_k(C_N)}{\binom{N}{k}} = \left\{\begin{array}{ll} 1 & \mbox{if } 1/2<\delta\le 1,\\[1mm]
(2\delta)^k & \mbox{if } 0\le \delta< 1/2.\end{array}\right.
\end{equation}

In (\ref{1A.2}) and (\ref{1A.3}), the asymptotic behavior changes abruptly when a certain threshold is passed. In (\ref{1A.4}), a change (from the constant $1$ to a limit depending on $\delta$) is still visible.

The question arises (of interest from a theoretical point of view, though irrelevant for applications as in \cite{DT10} or \cite{ALMT14}) what precisely happens at the thresholds. The first part of the following is concerned with this question. It turns out that here a monotonicity property of the quotient $\bE f_k(C_N)/\binom{N}{k}$ as a function of $N$ is useful, and this will be proved in Theorem \ref{T3A.1}. This monotonicity may be of independent interest. Here it allows us to give an elementary proof that the limit in (\ref{1A.4}) is equal to $1$ also for $\delta=\frac{1}{2}$. Curiously, none of the proofs given in \cite{HS20} for $\delta<1/2$ or $\delta>1/2$ extends to include the case $\delta=1/2$, so a different argument is required. A proof using deeper deviation results from probability theory was given by Godland, Kabluchko and Th\"ale \cite[Theorem 4.1]{GKT20}.

Using the monotonicity proved in Theorem \ref{T3A.1}, we can also supplement the asymptotic relations (\ref{1A.3}) for the Cover--Efron cones by a description of the case $\rho=\rho_W(\delta)$; see Theorem \ref{T5A.1}. For the Donoho--Tanner cones, corresponding results are an immediate consequence of Propositions 1 and 2 in Section \ref{sec5}, concerning the so-called Wendel probabilities.

Whereas (\ref{1A.2}) deals with the quotient $\bE f_k(D_N)/\binom{N}{k}$, also the difference $\binom{N}{k}-\bE f_k(D_N)$ exhibits a threshold phenomenon, though with a considerably smaller critical value. This will be treated in Section \ref{sec6}, also for the random cones $C_N$.

\section{Preliminaries}\label{sec2}

The expectations $\bE f_k(D_N)$, $\bE f_k(C_N)$ can be expressed in terms of the Wendel probabilities, which will be defined next. Let $X_1,\dots,X_N$ be as above. We define the probability
$$ P_{d,N} := \bP\left(o\notin {\rm conv}\{X_1,\dots,X_N\}\right)= \bP\left({\rm pos}\{X_1,\dots,X_N\}\not=\R^d\right),$$
where $o$ denotes the origin of $\R^d$. The second equality follows from the fact that almost surely $X_1,\dots,X_N$ are in general position and span $\R^d$ linearly. We call the numbers $P_{d,N}$ the {\em Wendel probabilities}, since they were first explicitly determined by Wendel \cite{Wen62} (the proof is reproduced in \cite[8.2.1]{SW08}). The result is that
\begin{equation}\label{2A.0}
P_{d,N}= \frac{C(N,d)}{2^N},\qquad C(N,d) := 2\sum_{i=0}^{d-1}\binom{N-1}{i}.
\end{equation}
Thus, if $\xi_N$ denotes a random variable which has the binomial distribution with parameters $N$ and $1/2$, so that
$$ \bP(\xi_N=k)=\frac{1}{2^N}\binom{N}{k} \quad \mbox{and} \quad \bP(\xi_N\le k) =\frac{1}{2^N}\sum_{i=0}^k\binom{N}{i},$$
then
\begin{equation}\label{2A.1}
P_{d,N} =\bP(\xi_{N-1}\le d-1).
\end{equation}

We have
\begin{equation}\label{2A.2}
\frac{\bE f_k(D_N)}{\binom{N}{k}} = P_{d-k,N-k},
\end{equation}
as shown by Donoho and Tanner \cite[Thm. 1.6]{DT10},
and
\begin{equation}\label{2A.3}
\frac{\bE f_k(C_N)}{\binom{N}{k}} = \frac{P_{d-k,N-k}}{P_{d,N}};
\end{equation}
see \cite[(3.3)]{CE67} or \cite[(27)]{HS16}. Together with (\ref{2A.1}), this is the reason why limit theorems and deviation results from probability theory can be applied to treat the asymptotic behavior of $\bE f_k(D_N)$ and $\bE f_k(C_N)$.

\section{Monotonicity of the quotient}\label{sec3}

In stochastic geometry, monotonicity of expected face numbers is not always so easy to predict or to establish as one might expect; we mention \cite{DGGMR13} as one example. In the following, we deal with the quotient $\bE f_k(C_N)/\binom{N}{k}$ as a function of $N$.

\begin{theorem}\label{T3A.1}
If $N>d>k$, then
$$ \frac{\bE f_k(C_{N+1})}{\binom{N+1}{k}} < \frac{\bE f_k(C_N)}{\binom{N}{k}}.$$
\end{theorem}

\begin{proof}
Let $d,N,k\in\N$ with $N>d>k$ be given. We define
$$ S_j:= \sum_{i=0}^{d-j-2} \binom{N-j-1}{i} \quad\mbox{for }j=0,\dots,k-1,$$
where $d$ and $N$ remain fixed and hence are not shown by the notation. Moreover, we set $S_j:=0$ for $j\ge d-1$. First we show that
\begin{equation}\label{5.0}
\frac{S_{j+1}}{\binom{N-j-2}{d-j-2}} < \frac{S_j}{\binom{N-j-1}{d-j-1}}
\end{equation}
for $j\le d-2$. For the proof, we use
\begin{equation}\label{5.1}
\sum_{i=0}^p \binom{M+1}{i} = 2\sum_{i=0}^{p-1}\binom{M}{i}+\binom{M}{p},
\end{equation}
with $M=N-j-2$ and $p=d-j-2$ , to get
\begin{eqnarray*}
S_j &=& \sum_{i=0}^{d-j-2} \binom{N-j-1}{i} = 2\sum_{i=0}^{d-j-3} \binom{N-j-2}{i} +\binom{N-j-2}{d-j-2} \\
&=& 2S_{j+1}+\binom{N-j-2}{d-j-2}.
\end{eqnarray*}
Thus, (\ref{5.0}) is equivalent to
$$ \frac{\binom{N-j-1}{d-j-1}}{\binom{N-j-2}{d-j-2}}S_{j+1} < 2S_{j+1}+\binom{N-j-2}{d-j-2},$$
which is equivalent to
$$
\frac{N-2d+j+1}{d-j-1}\, S_{j+1}<\binom{N-j-2}{d-j-2}.
$$
If $N-2d+j\le -1$, then this holds trivially. Hence, assume that $N-2d+j\ge 0$. Then the previous relation is equivalent to
\begin{equation}\label{5.4}
\frac{S_{j+1}}{\binom{N-j-2}{d-j-2}} < \frac{d-j-1}{N-2d+j+1}.
\end{equation}
From \cite[Lemma 3(a)]{HS20} with $n=N-j-2$ and $m=d-j-3$ if $d-j-2>0$ (the assumption $2m<n+1$ is satisfied since $N-2d+j\ge 0$) we obtain
$$ \frac{S_{j+1}}{\binom{N-j-2}{d-j-2}} \le \frac{d-j-2}{N-d+1} \cdot\frac{N-d+2}{N-2d+j+5} \le  \frac{d-j-2}{N-d+1} \cdot\frac{N-d+2}{N-2d+j+1}.$$
If $d-j-2=0$, this inequality holds trivially (note that $S_{j+1}=0$ in that case). Since $N-2d+j\ge 0$, the right side is strictly smaller than the right side of (\ref{5.4}). This proves (\ref{5.0}).

Now we recall from (\ref{2A.3}) and (\ref{2A.0}) that
$$ g_k(N):= \frac{\bE f_k(C_N)}{\binom{N}{k}}=2^k\frac{C(N-k,d-k)}{C(N,d)} = 2^k \frac{\sum_{i=0}^{d-k-1} \binom{N-k-1}{i}}{\sum_{i=0}^{d-1} \binom{N-1}{i}}.$$
We write this as
$$ g_k(N)= 2^k \frac{S_k+\binom{N-k-1}{d-k-1}}{S_0+\binom{N-1}{d-1}}.$$
Using (\ref{5.1}) again, we obtain
\begin{eqnarray*}
g_k(N+1) &=&2^k \frac{\sum_{i=0}^{d-k-1} \binom{N-k}{i}}{\sum_{i=0}^{d-1} \binom{N}{i}}= 2^k \frac{2\sum_{i=0}^{d-k-2} \binom{N-k-1}{i} + \binom{N-k-1}{d-k-1}}{2\sum_{i=0}^{d-2} \binom{N-1}{i} +\binom{N-1}{d-1}}\\
& = &2^k\frac{2S_k+\binom{N-k-1}{d-k-1}}{2S_0+\binom{N-1}{d-1}}
\end{eqnarray*}
(recalling that $S_k$ is defined with the fixed number $N$). Therefore, $g_k(N+1)<g_k(N)$ holds if and only if
\begin{equation}\label{5.2}
\frac{S_k}{\binom{N-k-1}{d-k-1}} < \frac{S_0}{\binom{N-1}{d-1}}.
\end{equation}
But this follows from the inequality (\ref{5.0}) by induction. The proof is complete.
\end{proof}

We remark that from $N\ge d+1$ and Theorem \ref{T3A.1} it follows that $g_k(N)\le g_k(d+1)$, hence
$$ \frac{\bE f_k(C_N)}{\binom{N}{k}} \le \frac{2^d-2^k}{2^d-1},$$
where equality holds if and only if $N=d+1$.

\section {The missing case in (\ref{1A.4}).}\label{sec4}

Under the assumption that $\frac{d}{N}\to\frac{1}{2}$ as $d\to\infty$ and that $k\in\N$ is fixed, we give an elementary proof for
\begin{equation}\label{1A.4.1}
\lim_{d\to\infty} \frac{\bE f_k(C_N)}{\binom{N}{k}} =1.
\end{equation}
A proof using more sophisticated tools from probability theory was given in \cite[Thm. 4.1]{GKT20}.

First we assume that $N<2d$ for all $d$. Then (\ref{1A.4.1}) holds by Theorem 2 in \cite{HS20}.

Now we assume that $N\ge 2d$ for all $d$. Let $0<\delta<1/2$ and choose integers $N_1=N_1(d)$ such that
$$ \frac{d}{N_1}\to\delta\quad\mbox{as }d\to\infty.$$
For sufficiently large $d$ we then have $\frac{d}{N}\ge \frac{d}{N_1}$ and hence $N\le N_1$. From Theorem \ref{T3A.1} and induction we get
$$ \frac{\bE f_k(C_N)}{\binom{N}{k}} \ge \frac{\bE f_k(C_{N_1})}{\binom{N_1}{k}}$$
and hence
$$ \liminf_{d\to\infty}\frac{\bE f_k(C_N)}{\binom{N}{k}} \ge \liminf_{d\to\infty}\frac{\bE f_k(C_{N_1})}{\binom{N_1}{k}} =(2\delta)^k,$$
the latter by (\ref{1A.4}). Since $\bE f_k(C_N)/\binom{N}{k}\le 1$ and $\delta<1/2$ was arbitrary, relation (\ref{1A.4.1}) follows.

Finally, for arbitrary $N=N(d)$ with $\frac{d}{N}\to\frac{1}{2}$, we consider the subsequence of all $N$ for which $N<2d$ and the subsequence of all $N$ for which $N\ge 2d$. Since either subsequence satisfies (\ref{1A.4.1}), the sequence itself satisfies (\ref{1A.4.1}).

\section {The critical case in (\ref{1A.3}).}\label{sec5}

In this section, we assume that (\ref{1A.1}) holds, and we ask for the asymptotic behavior of $\bE f_k(C_N)/\binom{N}{k}$ in the critical case $\rho=\rho_W(\delta)$, which was left out in (\ref{1A.3}).

The case $\delta<1/2$ and $\rho=0$ is not very interesting. If $k\to\infty$, it was remarked in \cite{HS20} after Theorem 1 that
$$ \lim_{d\to\infty} \frac{\bE f_k(C_N)}{\binom{N}{k}} =0.$$
If $k\to\infty$ is not required, we can choose a sequence $(k(d))_{d\in\N}$ that attains only two different values, each one infinitely often. Then it follows from (\ref{1A.4}) that the sequence $\left(\bE f_k(C_N)/\binom{N}{k}\right)_{d\in\N}$ has two subsequences with different limits, hence this sequence has no limit.

Now we consider the case
\begin{equation}\label{5A.1}
\delta>1/2 \quad\mbox{and} \quad \rho=\rho_W(\delta)= 2-1/\delta.
\end{equation}
Under the assumptions $N=\frac{d}{\delta}+c\sqrt{d}+o\left(\sqrt{d}\right)$ and $k=\left(2-\frac{1}{\delta}\right) d+ b\sqrt{d}+o\left(\sqrt{d}\right)$, Godland, Kabluchko and Th\"ale \cite[Thm. 4.4]{GKT20} have determined the limit for $\bE f_k(C_N)/\binom{N}{k}$. The following theorem improves and supplements their result.

If we assume that (\ref{1A.1}) holds and $\delta>1/2$, then $\rho=\rho_W(\delta)$ is equivalent to $\frac{N-2d+k}{d}\to 0$ as $d\to\infty$. Theorem \ref{T5A.1} describes how the asymptotic behavior of $N-2d+k$ is crucial. The comparison with $d^\alpha$, for $\alpha\in (0,1)$ is relevant, and it turns out that the exponent $\alpha=1/2$ is a new threshold.

In the following, if we write $f(d)\sim c\sqrt{d}$ for a constant $c\in\R$ (which may be zero), we mean that $f(d)/\sqrt{d}\to c$ as $d\to \infty$. By $\Phi$ we denote the distribution function of the standard normal distribution.

\begin{theorem}\label{T5A.1}
Suppose that $(\ref{1A.1})$ and $(\ref{5A.1})$ hold.

\noindent
{\rm (a)} If
\begin{equation}\label{5A.2}
N-2d+k \sim cd^{1/2} \quad\mbox{with }c\in\R,
\end{equation}
then
$$ \lim_{d\to\infty}  \frac{\bE f_k(C_N)}{\binom{N}{k}}=\Phi\left(-\frac{c}{\sqrt{2(1-\rho)}}\right)\in (0,1).$$
{\rm (b)} If
\begin{equation}\label{5A.3}
N-2d+k \ge cd^\alpha\quad\mbox{with }\frac{1}{2}<\alpha<1 \text{ and } c> 0,
\end{equation}
then
$$ \lim_{d\to\infty} \frac{\bE f_k( C_N)}{\binom{N}{k}} =0.$$
{\rm (c)} If
\begin{equation}\label{5A.4}
-cd^\alpha\le N-2d+k \le cd^\alpha\quad\mbox{with }0<\alpha<\frac{1}{2} \text{ and } c\ge 0,
\end{equation}
then
$$ \lim_{d\to\infty} \frac{\bE f_k( C_N)}{\binom{N}{k}} =\frac{1}{2}.$$
{\rm (d)} If
\begin{equation}\label{5A.5}
N-2d+k \le cd^\alpha\quad\mbox{with }\frac{1}{2}<\alpha<1 \text{ and } c< 0,
\end{equation}
then
$$ \lim_{d\to\infty}  \frac{\bE f_k( C_N)}{\binom{N}{k}} =1.$$
\end{theorem}

\begin{proof}
(a) Suppose that (\ref{5A.2}) holds. As observed in \cite{HS20} (proof of Theorem 2), it follows from the Berry--Esseen theorem that
\begin{align*}
\frac{\bE f_k(C_N)}{\binom{N}{k}}
&=\frac{\Phi\left(\frac{2d-N-k-1}{\sqrt{N-k-1}}\right)+O\left(\frac{1}{\sqrt{N-k-1}}\right)}{
\Phi\left(\frac{2d-N-1}{\sqrt{N-1}}\right)+O\left(\frac{1}{\sqrt{N-1}}\right)}.
\end{align*}
We have
$$ \frac{N-k-1}{d}\to \frac{1}{\delta}-\rho=2(1-\rho),$$
$$\frac{2d-N-k-1}{\sqrt{N-k-1}} =\frac{2d-N-k}{\sqrt{d}\sqrt{\frac{N-k-1}{d}}}-\frac{\frac{1}{\sqrt{d}}}{\sqrt{\frac{N-k-1}{d}}} \to \frac{-c}{\sqrt{2(1-\rho)}},$$
$$\frac{2d-N-1}{\sqrt{N-1}} = \frac{\sqrt{d}\left(2-\frac{N}{d}-\frac{1}{d}\right)}{\sqrt{\frac{N-1}{d}}}  \to\infty,$$
where we have used that $\delta>1/2$.
The assertion follows.

(b) Suppose that (\ref{5A.3}) holds. For any given number $c_1\in\R$ we can choose a sequence $(N_1(d))_{d\in\N}$ of integers such that $N_1=N_1(d)$ satisfies $N_1-2d+k\sim c_1d^{1/2}$ and $N_1<N$ for sufficiently large $d$. From (a) it follows that $\lim_{d\to\infty} (\bE f_k(C_N)/\binom{N}{k}) =\Phi(-c_1/\sqrt{2(1-\rho)})$. The monotonicity proved in Theorem \ref{T3A.1} now yields
$$ \limsup_{d\to\infty} \frac{\bE f_k(C_N)}{\binom{N}{k}} = \Phi\left(\frac{-c_1}{\sqrt{2(1-\rho)}}\right).$$
Since this holds for all $c_1\in \R$, the assertion of (b) follows.

(c) Suppose that (\ref{5A.4}) holds. Then we have
$$  \frac{|N-2d+k|}{\sqrt{d}} \le c d^{\alpha-\frac{1}{2}} \to 0$$
as $d\to\infty$. Hence $N-2d+k\sim 0\sqrt{d}$ and the assertion follows
from (a), since
$$ \lim_{d\to\infty} \frac{\bE f_k(C_N)}{\binom{N}{k}} =\Phi(0)=\frac{1}{2}.$$

(d) Suppose that (\ref{5A.5}) holds. Since $(N-2d+k)/d^{1/2}\to-\infty$, for any given number $c_2\in\R$ we can chose a sequence $(N_2(d))_{d\in\N}$ of integers such that $N_2=N_2(d)$ satisfies $(N_2-2d+k)/d^{1/2}\to c_2$ and $N_2>N$ for sufficiently large $d$. From (a) it follows that
$$ \liminf_{d\to\infty} \frac{\bE f_k(C_N)}{\binom{N}{k}} \ge \Phi\left(-c_2/\sqrt{2(1-\rho)}\right).$$
Since this holds for all $c_2\in \R$, we get the assertion of (d).
\end{proof}

Finally, if $\delta=\frac{1}{2}$ and $\rho=0$, then by suitable choices of sequences $N(d),k(d)$, any limit in $[0,1]$ can be achieved for $\bE f_k(C_N)/\binom{N}{k}$, and (hence) in general no limit exists. This can be deduced from the fact, following from the proof of part (a) of Theorem \ref{T5A.1} and also shown in \cite[Thm 4.4]{GKT20}, that $N-2d\sim c\sqrt{d}$ and $k\sim b\sqrt{d}$ together imply
$$ \lim_{d\to\infty} \frac{\bE f_k(C_N)}{\binom{N}{k}} =\frac{\Phi\left(-(c+b)/\sqrt{2}\right)}{\Phi\left(-c/\sqrt{2}\right)}.$$

\vspace{3mm}

Similar results for the Donoho--Tanner random cones $D_N$ follow immediately from corresponding results for the Wendel probabilities. We formulate these here for completeness.

\begin{proposition}\label{T2.1}
Suppose that
$$ \frac{d}{N} \to\delta \quad\mbox{as }d\to\infty,$$
with a number $\delta\in[0,1]$. Then the Wendel probabilities satisfy
$$ \lim_{d\to\infty} P_{d,N}= \left\{\begin{array}{ll} 1 & \mbox{if } 1/2<\delta\le1,\\[1mm]
0 & \mbox{if } 0\le \delta< 1/2.\end{array}\right.$$
Moreover, the convergence is exponentially fast.
\end{proposition}

\begin{proof}
This follows from the Remark in \cite{HS20}, after the proof of Theorem 7. For the reader's convenience, we reproduce the short proof here. Let $\delta>\frac{1}{2}$. If $d$ is sufficiently large, we have $\frac{d}{N-1}-\frac{1}{2}\ge \frac{1}{\sqrt{2}}\left(\delta-\frac{1}{2}\right)>0$. Therefore,
\begin{eqnarray*}
\bP(\xi_{N-1}\ge d) &=& \bP\left(\frac{\xi_{N-1}}{N-1}-\frac{1}{2}\ge \frac{d}{N-1}-\frac{1}{2}\right)\\
&\le& \exp\left(-2\left(\frac{d}{N-1}-\frac{1}{2}\right)^2(N-1)\right)
\end{eqnarray*}
by Okamoto \cite[Theorem  2(i)]{Oka58} (which can be applied since $\frac{d}{N-1}-\frac{1}{2}>0$). It follows that
$$ P_{d,N} =\bP(\xi_{N-1}\le d-1)\ge 1-\exp\left(-\left(\delta-\frac{1}{2}\right)^2(N-1)\right)\to 1$$
as $d\to\infty$.

Let $\delta<\frac{1}{2}$. If $d$ is sufficiently large, we have $\frac{1}{2}-\frac{d-1}{N-1}\ge \frac{1}{\sqrt{2}}\left(\frac{1}{2}-\delta\right)>0$. Therefore,
\begin{eqnarray*}
\bP(\xi_{N-1}\le d-1) &=& \bP\left(\frac{\xi_{N-1}}{N-1}-\frac{1}{2}\le -\left(\frac{1}{2}-\frac{d-1}{N-1}\right)\right)\\
&\le& \exp\left(-2\left(\frac{d-1}{N-1}-\frac{1}{2}\right)^2(N-1)\right)
\end{eqnarray*}
by \cite[Theorem  2(ii)]{Oka58}. It follows that
$$ P_{d,N} =\bP(\xi_{N-1}\le d-1)\le \exp\left(-\left(\delta-\frac{1}{2}\right)^2(N-1)\right)\to 0$$
as $d\to\infty$.
\end{proof}

The critical case $\delta=1/2$ is covered by the following proposition.

\begin{proposition}\label{T2.2}
Let $c$ be a real constant.

\noindent
{\rm (a)} If
\begin{equation}\label{2.a}
N-2d\sim cd^{1/2},
\end{equation}
then
$$ \lim_{d\to\infty} P_{d,N}= \Phi\left(-c/\sqrt{2}\right)\in (0,1).$$
{\rm (b)} If
\begin{equation}\label{2.2}
N-2d \ge c d^\alpha\quad\mbox{with }\frac{1}{2}<\alpha<1 \text{ and } c> 0,
\end{equation}
then
$$ \lim_{d\to\infty} P_{d,N} =0.$$
{\rm (c)} If
\begin{equation}\label{2.1}
-cd^\alpha\le N-2d \le cd^\alpha\quad\mbox{with }0<\alpha<\frac{1}{2} \text{ and } c\ge 0,
\end{equation}
then
$$ \lim_{d\to\infty} P_{d,N} =\frac{1}{2}.$$
{\rm (d)} If
\begin{equation}\label{2.2b}
N-2d \le cd^\alpha\quad\mbox{with }\frac{1}{2}<\alpha<1 \text{ and } c< 0,
\end{equation}
then
$$ \lim_{d\to\infty} P_{d,N} =1.$$

In particular, for $1/2<\alpha<1$ we can choose $N=N(d)$ in such a way that $\frac{d}{N}\to\frac{1}{2}$,
$$ N-2d=O(d^\alpha),$$
and
$$ \liminf_{d\to\infty} P_{d,N} = 0,\qquad \limsup_{d\to\infty} P_{d,N}= 1.$$
\end{proposition}

\begin{proof}
(a) Let (\ref{2.a}) be satisfied. As in the proof of Theorem 2 in \cite{HS20}, we have
$$ P_{d,N} =\Phi\left(\frac{2d-N-1}{\sqrt{N-1}}\right)+O\left(\frac{1}{\sqrt{N-1}}\right).$$
The assumption (\ref{2.a}) implies that
$$\frac{2d-N-1}{\sqrt{N-1}}\to -\frac{c}{\sqrt{2}}\quad\mbox{as }d\to\infty,$$
so that the assertion follows.

Assertions (b), (c), (d) follow from (a) in combination with the trivial monotonicity, $P_{d,N_1} > P_{d,N}$ if $N_1<N$, similarly as in the proof of Theorem \ref{T5A.1}.

The remaining assertion follows by choosing a sequence $(N(d))_{d\in\N}$ with suitable subsequences.
\end{proof}

As mentioned in (\ref{2A.2}), we have
\begin{equation}\label{3.1}
\frac{\bE f_k(D_N)}{\binom{N}{k}} = P_{d-k,N-k}.
\end{equation}
Let us first remark that Proposition \ref{T2.1} allows us to give a very short proof of (\ref{1A.2}). In fact, under the assumptions of (\ref{1A.2}), we have $d-k\to \infty$ as $d\to\infty$, since $\rho<1$. Further, the limit
$$ \lim_{d\to\infty} \frac{d-k}{N-k} = \frac{1-\rho}{1/\delta-\rho}$$
is larger (smaller) than $1/2$ if $\rho<\rho_W(\delta)$ (respectively,  $\rho>\rho_W(\delta)$). Hence, Proposition \ref{T2.1} implies (\ref{1A.2}) immediately.

In view of (\ref{3.1}), a counterpart to Theorem \ref{T5A.1} for $D_N$ can immediately be derived from Proposition \ref{T2.2}. We refrain from stating it explicitly.

Further, we remark that under the assumptions
$$ \frac{d}{N}\to\delta\quad\mbox{as }d\to\infty,\qquad k \mbox{ fixed},$$
we have
$$ \lim_{d\to\infty} \frac{\bE f_k(D_N)}{\binom{N}{k}} =\left\{ \begin{array}{ll} 1 &\mbox{if } 1/2<\delta\le 1,\\ 1/2 & \mbox{if } \delta=1/2,\\ 0 & \mbox{if }0\le \delta<1/2.\end{array}\right.$$
The cases $1/2<\delta$ and $\delta< 1/2$ follow from Proposition \ref{T2.1}, and the case $\delta=1/2$ follows from Proposition \ref{T2.2}(a) with $c=0$.

\section{Differences instead of quotients}\label{sec6}

The asymptotic relation (\ref{1A.2}) deals with the quotient $\bE f_k(D_N)/\binom{N}{k}$. Instead of this, we can also consider the difference $\binom{N}{k}- \bE f_k(D_N)$. Here a smaller critical number appears, which first showed up in related work of Donoho and Tanner \cite{DT10}. To recall it, we denote by ${\sf H}$ the binary entropy function with base $e$,
$$ {\sf H}(x):= -x\log x-(1-x)\log(1-x)\quad\mbox{for } 0\le x\le 1,$$
and consider the function
$$ G(\delta,\rho) := {\sf H}(\delta)+\delta {\sf H}(\rho)-(1-\rho\delta)\log 2 \quad\mbox{for } \rho,\delta\in[0,1],$$
introduced by Donoho and Tanner (with different notation). For fixed $\delta\in(1/2,1)$, the function $G_\delta$ defined by $G_\delta(x):= G(\delta,x)$ has a unique zero in $[0,1]$ (see, e.g., Lemma 6 in \cite{HS20}), which is denoted by $\rho_S(\delta)$.

We remark that Donoho and Tanner \cite{DT10} have used Boole's inequality to show that
$$ \bP\left(f_k(D_N)=\binom{N}{k}\right) \ge 1-\left[\binom{N}{k}-\bE f_k(D_N)\right],$$
and have proved that the right-hand side tends to $1$ as $d\to\infty$ if $\rho<\rho_S(\delta)$. The following theorem improves the latter fact and shows that for the difference $\binom{N}{k}- \bE f_k(D_N)$, the number $\rho_S(\delta)$ is, in fact, a threshold. We say that a function $g$ depending on the dimension $d$ remains in an interval $I$ if $g(d)\in I$ for all $d$. An interval is called positive if its bounds are positive numbers.

\begin{theorem}\label{T6.1}
Let $1/2<\delta<1$ and $0<\rho<\rho_W(\delta) =2-1/\delta$ be given. Let $k<d<N$ be integers (depending on $d$) such that
$$ \frac{d}{N}=:\delta_d\to\delta,\qquad \frac{k}{d}=:\rho_d\to \rho\qquad\mbox{as } d\to\infty.$$
Then
\begin{equation}\label{6.0}
\binom{N}{k}- \bE f_k(D_N) =\frac{1}{N} e^{NG(\delta_d,\rho_d)}\cdot g(d),
\end{equation}
where $g$ remains in a bounded positive interval.

In particular,
$$ \lim_{d\to\infty} \left[\binom{N}{k}-\bE f_k(D_N)\right] =\left\{\begin{array}{ll} 0 & \mbox{if }\rho <\rho_S(\delta),\\[1mm]
\infty & \mbox{if }\rho>\rho_S(\delta).\end{array}\right.$$

If $\rho=\rho_S(\delta)$, then the numbers $N$ and $k$  can be chosen such that
$$ \liminf_{d\to\infty} \left[\binom{N}{k}- \bE f_k(D_N)\right]=0\quad\mbox{and} \quad  \limsup_{d\to\infty} \left[\binom{N}{k}- \bE f_k(D_N)\right]=\infty.$$
\end{theorem}

\begin{proof}
From (\ref{3.1}) we have
\begin{equation}\label{6.1}
\binom{N}{k}-\bE f_k(D_N)= \binom{N}{k}\left[1-P_{d-k,N-k}\right]=\binom{N}{k}P_{N-d,N-k},
\end{equation}
since $P_{m,M}+P_{M-m,M}=1$, hence
\begin{equation}\label{6.2}
\binom{N}{k}-\bE f_k(D_N) =  \binom{N}{k} \frac{1}{2^{N-k-1}} \sum_{i=0}^{N-d-1} \binom{N-k-1}{i}.
\end{equation}
We consider each of the three factors. Abbreviating $\delta_d\rho_d=:\tau_d$, we obtain from Stirling's approximation
$$
 n!= \sqrt{2\pi n}\,e^{-n}n^ne^{\theta/12n},\quad 0<\theta=\theta(n)<1,
$$
that
\begin{equation}\label{6.3}
\binom{N}{k} =\frac{1}{\sqrt{2\pi \tau_d(1-\tau_d)}}\cdot \frac{1}{\sqrt{N}}\cdot e^{N {\sf H}(\tau_d)}\cdot e^{\frac{\theta_1}{12N}},
\end{equation}
with a number $\theta_1$ (depending on $d$) in a fixed interval independent of $d$. Next, we write
\begin{equation}\label{6.4}
\frac{1}{2^{N-k-1}} =2e^{-N(1-\tau_d)\log 2}.
\end{equation}
To estimate the sum of binomial coefficients, we note that Lemma 3 of \cite{HS20} yields the following. If $n,m\in\N$ are integers with $2m<n+1$, then
$$ \binom{n}{m} \le\sum_{i=0}^m \binom{n}{i} \le\binom{n}{m}\frac{n-m+1}{n-2m+1}.$$
With $m=N-d-1$ and $n=N-k-1$ we have $2m<n$ for sufficiently large $d$, since
$$ \frac{N-d-1}{N-k-1} < \frac{N-d}{N-k}=:\sigma_d =\frac{1-\delta_d}{1-\tau_d} \to \frac{1-\delta}{1-\delta\rho} <\frac{1}{2}.$$
Since
$$ \frac{n-m+1}{n-2m+1} =\frac{d-k+1}{2d-N-k+2} \to \frac{1-\rho}{\rho_W(\delta)-\rho}$$
as $d\to\infty$ and
$$
\binom{N-k-1}{N-d-1}=\binom{N-k}{N-d}\sigma_d,
$$
we see that
$$\sum_{i=0}^{N-d-1} \binom{N-k-1}{i} =\binom{N-k}{N-d}\cdot g_1(d),$$
where $g_1$ remains in a bounded positive interval. Stirling's approximation gives
\begin{equation}\label{6.5}
\binom{N-k}{N-d} =\frac{1}{\sqrt{2\pi \sigma_d(1-\sigma_d)}} \cdot \frac{1}{\sqrt{N-k}}\cdot e^{(N-k) {\sf H}(\sigma_d)}\cdot e^{\frac{\theta_2}{12(N-k)}}
\end{equation}
with a number $\theta_2$ in a fixed interval independent of $d$. We combine (\ref{6.3}), (\ref{6.4}), (\ref{6.5}) and note that each of
$$ \frac{e^{\theta_1/12N}}{\sqrt{2\pi\tau_d(1-\tau_d)}}, \qquad \frac{e^{\theta_2/12(N-k)}}{\sqrt{2\pi\sigma_d(1-\sigma_d)}},\qquad \sqrt{\frac{N}{N-k}}$$
has a positive finite limit as $d\to\infty$. It follows that
$$ \binom{N}{k} -\bE f_k(D_N) = \frac{1}{N} e^{N[{\sf H}(\tau_d)+(1-\tau_d){\sf H}(\sigma_d)-(1-\tau_d)\log 2]}\cdot g(d),$$
where $g$ remains in a bounded positive interval.

Since an elementary calculation yields that
\begin{equation}\label{6.5a}
{\sf H}(\tau_d) +(1-\tau_d){\sf H}(\sigma_d) -(1-\tau_d)\log 2 = G(\delta_d,\rho_d),
\end{equation}
relation (\ref{6.0}) follows.

Now suppose that $\rho<\rho_S(\delta)$. Since $G(\delta_d,\rho_d) \to G(\delta,\rho) <0$, there is a constant $c$ such that $G(\delta_d,\rho_d) \le c<0$ for all large $d$. Hence, (\ref{6.0}) implies that $\lim_{d\to\infty} \left[\binom{N}{k} -\bE f_k(D_N)\right] =0$. This was proved in a different way in \cite{DT10}.

The assertion that  $\lim_{d\to\infty} \left[\binom{N}{k} -\bE f_k(D_N)\right] =\infty$ if $\rho>\rho_S(\delta)$ (which implies that $G(\delta,\rho)>0$) is now obtained in the same way.

Now we assume that $\rho=\rho_S(\delta)$. To prove the remaining assertion, we start with the fixed $\delta\in(1/2,1)$ and define $N=N(d)$ by
$$
N:= \lfloor d/\delta\rfloor, \qquad\mbox{hence } 0\le\delta_d-\delta<\frac{1}{N}.
$$
The function $G(\delta,\cdot)$ is equal to $0$ at $\rho$ and is strictly increasing in a neighborhood of $\rho$. Since $0<\left((-1)^d\frac{1}{2}+1\right)  N^{-1}\cdot \log N\to 0$ as $d\to\infty$, for sufficiently large $d$ there is a unique number $\rho(d)\in (\rho,1)$ such that
$$ G(\delta,\rho(d)) =\left((-1)^d\frac{1}{2}+1\right)\cdot\frac{\log N}{N}.$$
We define the number $k=k(d)$ by
$$  k:= \lfloor \rho(d)d\rfloor, \qquad\mbox{hence } 0\le\rho(d)-\rho_d<\frac{2}{N},$$
since we can assume that $d/N>1/2$ if $d$ is sufficiently large.

The function $G$ is differentiable in a neighborhood of $(\delta,\rho)$, hence there are positive constants $c_1,c_2,c_3$, independent of $d$, such that for sufficiently large $d$ we have
\begin{eqnarray*}
|G(\delta_d,\rho_d)-G(\delta,\rho(d))| &\le & |G(\delta_d,\rho_d)-G(\delta,\rho_d)| + |G(\delta,\rho_d)-G(\delta,\rho(d))|\\
&\le& c_1|\delta_d-\delta|+c_2 |\rho_d-\rho(d)| \le \frac{c_3}{N}.
\end{eqnarray*}
Thus, if we write
$$ G(\delta_d,\rho_d)-G(\delta,\rho(d))= \frac{g_2}{N},$$
then $g_2$ is a function which remains in a bounded interval (independent of $d$), and
\begin{equation}\label{combi}
 \frac{1}{N} e^{N G(\delta_d,\rho_d)}=  e^{g_2}\cdot \begin{cases}
\sqrt{N},&\text{if $d$ is even},\\
\sqrt{N}^{-1},&\text{if $d$ is odd}.
\end{cases}
\end{equation}
Combining (\ref{6.0}) and \eqref{combi},  we obtain the remaining assertion.
\end{proof}

Now we deal with the Cover--Efron cones.

\begin{theorem}\label{T6.2}
Let $1/2<\delta<1$ and $0<\rho<\rho_W(\delta) =2-1/\delta$ be given. Let $k<d<N$ be integers (depending on $d$) such that
$$ \frac{d}{N}=:\delta_d\to\delta,\qquad \frac{k}{d}=:\rho_d\to \rho\qquad\mbox{as } d\to\infty.$$
Then
\begin{equation}\label{6.6}
\binom{N}{k}- \bE f_k(C_N) = \frac{1}{N}e^{NG(\delta_d,\rho_d)}\cdot h(d),
\end{equation}
where $h$ remains in a bounded positive interval.

In particular,
$$ \lim_{d\to\infty} \left[\binom{N}{k}-\bE f_k(C_N)\right] =\left\{\begin{array}{ll} 0 & \mbox{if }\rho <\rho_S(\delta),\\[1mm]
\infty & \mbox{if }\rho>\rho_S(\delta).\end{array}\right.$$

If $\rho=\rho_S(\delta)$, then the numbers $N$ and $k$  can be chosen such that
$$ \liminf_{d\to\infty} \left[\binom{N}{k}- \bE f_k(C_N)\right]=0\quad\mbox{and} \quad  \limsup_{d\to\infty} \left[\binom{N}{k}- \bE f_k(C_N)\right]=\infty.$$
\end{theorem}

\begin{proof}
It follows from (\ref{2A.2}) and (\ref{2A.3}) that $\binom{N}{k}-\bE f_k(C_N) < \binom{N}{k}-\bE f_k(D_N)$, hence Theorem \ref{T6.1} yields $\binom{N}{k}-\bE f_k(C_N)<\frac{1}{N}e^{NG(\delta_d,\rho_d)}\cdot c$ with a constant $c$ independent of $d$.

For a corresponding lower bound, we note that Lemma 1 in \cite{HS20} yields that
\begin{align*}
\binom{N}{k}-\bE f_k(C_N)&=\binom{N}{k}\frac{A}{1+A}\\
&=\binom{N}{k}\frac{1}{2^{N-1} P_{d,N}}\sum_{j=1}^{k}\binom{k}{j}\sum_{m=0}^{j-1}\binom{N-k-1}{d-k+m}.
\end{align*}
Using the rough estimates $P_{d,N}\le 1$ and
$$ \sum_{m=0}^{j-1}\binom{N-k-1}{d-k+m}\ge \binom{N-k-1}{d-k},$$
we obtain
\begin{align*}
\binom{N}{k}-\bE f_k(C_N) &\ge 2^{1-N}\binom{N}{k}\sum_{j=1}^k\binom{k}{j}\binom{N-k-1}{d-k}\\
&\ge 2^{k-N}\binom{N}{k}\frac{N-d}{N-k}\binom{N-k}{N-d}.
\end{align*}
We can now use (\ref{6.3}), (\ref{6.4}), (\ref{6.5}), (\ref{6.5a}) to complete the proof of (\ref{6.6}).

The last assertion follows as in the proof of Theorem \ref{T6.1}.
\end{proof}

\bigskip

\noindent
\textbf{Acknowledgement.} 
DH was supported by the German Research Foundation (DFG) via the Priority Program SPP 2265: \textit{Random Geometric Systems}.


\noindent Authors' addresses:\\[2mm]
Daniel Hug\\Karlsruhe Institute of Technology (KIT)\\
Department of Mathematics\\D-76128 Karlsruhe,
Germany\\
E-mail: daniel.hug@kit.edu\\[5mm]
Rolf Schneider\\Mathematisches Institut\\
Albert-Ludwigs-Universit{\"a}t\\
D-79104 Freiburg i.~Br., Germany\\
E-mail: rolf.schneider@math.uni-freiburg.de
\vfill

\end{document}